\documentclass[graybox]{svmult}

\usepackage{type1cm}         
\usepackage{makeidx}         
\usepackage{graphicx}        
\usepackage{multicol}        
\usepackage[bottom]{footmisc}

\usepackage{newtxtext}       %
\usepackage{newtxmath}       

\usepackage{amsmath}
\usepackage{tikz}
\usepackage{url}
\usetikzlibrary{arrows}

\makeindex             

\begin{document}

\title*{An inexact Petrov-Galerkin approximation for gas transport in pipeline networks}
\author{Herbert Egger \and Thomas Kugler \and Vsevolod Shashkov}
\institute{H. Egger \at TU Darmstadt, Department of Mathematics, \email{egger@mathematik.tu-darmstadt.de}
\and T. Kugler \at TU Darmstadt, Department of Mathematics, \email{kugler@mathematik.tu-darmstadt.de}
\and V. Shashkov \at TU Darmstadt, GSC Computational Engineering, \email{shashkov@gsc.tu-darmstadt.de}}
%
%
\maketitle

\def\myabstract{This paper studies the discretization of gas transport in pipeline networks by an inexact Petrov-Galerkin method. A full convergence analysis is presented for single pipes under the assumption of a linear friction law and the possible extension to pipe networks is discussed. The generalization to nonlinear gas transport models and the efficient implementation by hybridization is investigated numerically.}

\abstract*{\myabstract}

\abstract{\myabstract}

\def\dx{\partial_x}
\def\dt{\partial_t}
\def\dtt{\partial_{tt}}
\def\ddt{\frac{d}{dt}}
\def\RR{\mathbb{R}}
\def\V{\mathcal{V}}
\def\E{\mathcal{E}}
\def\div{\text{\rm div}}

\section{Introduction}
\label{sec:introduction}

The flow of gas in a horizontal pipeline of constant cross section is described by \cite{BrouwerGasserHerty11}
\begin{align}
A \dt \rho + \dx m &= 0 \label{eq:euler1}\\
\dt m + \dx \left(\frac{m^2}{A \rho} + A p\right) &= -\frac{\lambda}{2D} \frac{|m|}{A \rho} m. \label{eq:euler2}
\end{align}
Here $A$ and $D$ are the cross section and diameter of the pipe, and $\lambda$ is a dimensionless friction parameter. The functions $\rho$, $p$, and $m$ describe the density, pressure, and mass flow rate of the gas. 
Under isothermal flow conditions, one has
\begin{align}
p = c^2 \rho \label{eq:euler3}  
\end{align}
with constant $c$ denoting the speed of sound. In practically relevant scaling regimes, the nonlinear term on the 
left hand side of \eqref{eq:euler2} is usually neglected, which can be justified by an asymptotic analysis \cite{BrouwerGasserHerty11,MarcatiMilani90}.
Using this simplification and equation \eqref{eq:euler3} to eliminate the density, one arrives 
at evolution problems of the general form
\begin{align}
a \dt p + \dx m &= 0 \label{eq:sys1} \\
b \dt m + \dx p &= -d m \label{eq:sys2}
\end{align}
where $a$ and $b$ are positive constants and $d=d(p,m)$ denotes a state dependent friction coefficient. 
For our analyis, we will consider $d=d(x)$ as a function depending only on space which can be 
justified, e.g., by linearization around a steady state. 
Corresponding models for the gas flow on pipe networks are obtained by 
coupling the flow equations for single pipes via algebraic conditions \cite{RachfordDupont74,Reigstad14}; see below.

The discretization of \eqref{eq:sys1}--\eqref{eq:sys2} and its extension to pipeline networks 
has been discussed intensively in the literature. 
In \cite{RachfordDupont74}, a Galerkin approximation for \eqref{eq:euler1}--\eqref{eq:euler2}
with cubic Hermite polynomials is investigated numerically. The discretization of transient gas flow models is also studied \cite{BrouwerGasserHerty11,Osiadacz84}. An entropy stable finite volume method is proposed in  \cite{Reigstad14}, and an energy stable mixed finite element approximation is investigated in \cite{Egger18}. 
Apart from \cite{RachfordDupont74}, all methods discussed above are of lowest order 
and no rigorous convergence analysis is given. 
%
%

In this paper, we study the discretization of \eqref{eq:sys1}--\eqref{eq:sys2} by 
a Petrov-Galerkin approach of potentially high order. The resulting scheme is shown to be stable which allows 
us to prove order optimal convergence rates.
By using an appropriate functional analytic setting, the convergence results can be generalized 
almost verbatim to pipeline networks. A hybridization strategy will be discussed that facilitates 
the implementation and that allows to incorporate non-standard coupling conditions. 
The proposed method formally also allow to treat nonlinear models of gas transport and,
in principle, high order convergence can be obtained in practically relevant regimes.

\section{Notation and preliminaries}
\label{sec:2}

Let $x_L < x_R$ and denote by $L^p(x_L,x_R)$ and $W^{k,p}(x_L,x_R)$, $k \ge 0$ the standard Lebesgue and Sobolev spaces.
The scalar product and norm of $L^2(x_L,x_R)$ are written as $(v,w)$ and $\|v\|=\|v\|_{L^2}$. 
Other norms will be designated by subscripts.
We write $H^k(x_L,x_R)=W^{k,2}(x_L,x_R)$ for the Hilbert spaces and define 
\begin{align*}
H^1_0=\{v \in H^1(x_L,x_R) : v(x_L)=v(x_R)=0\} 
\quad \text{and} \quad H(\div) = H^1(x_L,x_R)
\end{align*}
for convenience. 
%
By $L^p(0,T;X)$ and $W^{k,p}(0,T;X)$ we denote the Bochner spaces of functions $f : [0,T] \to X$ with values in $X$. The value of $f(t)$ may then itself be a function. 
In the following, we consider the linear system of differential equations 
\begin{alignat}{2}
a \dt p(x,t) + \dx m(x,t) &= f(x,t),                 \label{eq:prob1}\\
b \dt m(x,t) + \dx p(x,t) + d(x) m(x,t) &= g(x,t),   \label{eq:prob2}
\end{alignat}
for $x_L < x < x_R$ and $t>0$ with homogeneous boundary conditions
\begin{alignat}{2}
p(x_L,t) &= p(x_R,t) = 0. \label{eq:prob3}
\end{alignat}
Inhomogeneous and more general boundary conditions can be considered as well and 
our analysis applies with minor modifications.
We will assume that 
\begin{itemize}
 \item[(A1)] $\quad \ a$, $b$ are positive constants, and 
 \item[(A2)] $\quad \ d \in L^\infty(x_L,x_R)$ with $0 < \underline d \le d(x) \le \overline d$ and constants $\underline d,\overline d$.
\end{itemize}
For given $f,g \in L^2(0,T;L^2(x_L,x_R))$ and
initial values $p(0) \in H^1_0$, $m(0) \in H(\div)$, existence of a unique solution 
follows from semigroup theory. Any smooth solution of problem \eqref{eq:prob1}--\eqref{eq:prob3} 
also satisfies $p(t) \in H^1_0$, $m(t) \in H(\div)$, and 
\begin{align}
(a \dt p(t),\widetilde q) + (\dx m(t),\widetilde q) &= (f(t),\widetilde q) \label{eq:var1}\\
(b \dt m(t),\widetilde v) + (\dx p(t),\widetilde v) + (d m(t),\widetilde v) &= (g(t),\widetilde v) \label{eq:var2}
\end{align}
for all $\widetilde v,\widetilde q \in L^2(x_L,x_R)$ and all $0 < t < T$. 
This variational characterization will 
be the starting point for our discretization approach introduced in the next section.

\section{Petrov-Galerkin approximation}
\label{sec:petrovgalerkin}

Let $x_L=x_0 < x_1 < \ldots < x_N=x_R$ be a partition of the interval $[x_L,x_R]$ into elements 
$T_n=[x_{n-1},x_n]$. We call $T_h := \{T_n : 1 \le n \le N\}$ the mesh and denote by $h_n = |x_{n}-x_{n-1}|$ and $h=\max_{n} h_n$ the local and global mesh size, respectively. 
Let 
\begin{align}
P_k(T_h) := \{v \in L^2(x_L,x_R) : v|_T \in P_k(T) \ \forall T \in T_h\} 
\end{align}
be the space of piecewise polynomials on the mesh $T_h$. 
We fix $k \ge 1$ and search for approximations for the solutions $p(t)$, $m(t)$ of problem \eqref{eq:prob1}--\eqref{eq:prob3} in the spaces
\begin{align}
Q_h = P_k(T_h) \cap H^1_0
\quad \text{and} \quad 
V_h = P_k(T_h) \cap H(\div) \label{eq:QhVh}
\end{align}
of continuous piecewise polynomials with appropriate boundary conditions. 
As finite dimensional test spaces for the variational problem \eqref{eq:var1}--\eqref{eq:var2}, we choose 
\begin{align}
\widetilde Q_h = P_{k-1}(T_h) 
\quad \text{and} \quad  
\widetilde V_h = P_{k-1}(T_h) \label{eq:QhtVht}
\end{align}
consisting of discontinuous piecewise polynomials of lower order $k-1$.
We denote by $I_h^k : H^1(x_L,x_R) \to P_k(T_h) \cap H^1(x_L,x_R)$ the $H^1$-projection operator, defined by
\begin{alignat}{2}
(I_h^k v)(x_{k}) &= v(x_{k}) && \text{for all }  0 \le k \le N,  \\
\text{and} \qquad 
(\dx I_h^k v,\widetilde v_h) &= (\dx v, \widetilde v_h) \qquad && \text{for all } \widetilde v_h \in P_{k-1}(T_h),
\end{alignat}
and let $\pi_h^{k-1} : L^2(x_L,x_R) \to P_{k-1}(T_h)$ be the $L^2$-orthogonal projection, satisfying 
\begin{alignat}{2}
(\pi_h^{k-1} v,\widetilde v_h) &= (v,\widetilde v_h) \qquad && \text{for all } \widetilde v_h \in P_{k-1}(T_h). 
\end{alignat}
Note that both projection operators $I_h^k$ and $\pi_h^{k-1}$ can be defined locally on every element. 
Moreover, they are mutually related other by the {\em commuting diagram} property
\begin{align} \label{eq:commuting}
\dx I_h^k v = \pi_h^{k-1} \dx v \qquad \text{for all } v \in H^1(x_L,x_R). 
\end{align}
For the approximation of problem \eqref{eq:prob1}--\eqref{eq:prob3}, we then use the following 
approximation. 
\begin{problem} \label{prob:semi}%
{\bf (Inexact Petrov-Galerkin method)} 
Find functions $p_h \in H_0^1(0,T;Q_h)$, $m_h \in H^1(0,T;V_h)$ with $p_h(0)=I_h^k p(0)$ and $m_h(0)=I_h^k m(0)$, 
and such that
\begin{align}
(a \dt p_h(t),\widetilde q_h) + (\dx m_h(t),\widetilde q_h) &= (f(t),\widetilde q_h)                \label{eq:var1h}\\
(b \dt m_h(t),\widetilde v_h) + (\dx p_h(t),\widetilde v_h) + (d \pi_h^{k-1} m_h(t),\widetilde v_h) &= (g(t),\widetilde v_h) \label{eq:var2h}
\end{align}
for all $\widetilde q_h \in \widetilde Q_h = P_{k-1}(T_h)$ and $\widetilde v_h \in \widetilde V_h = P_{k-1}(T_h)$, and for all $0 \le t \le T$.
\end{problem}
%
%
The well-posedness of this problem follows from the results of the next section.

\section{Discrete stability estimates}
\label{sec:stability}

We now derive some discrete stability estimates that yield well-posedness of the semidiscrete method 
and that allow us to establish error estimates of optimal order.

\begin{lemma} \label{lem:estimate}
Let $p_h$, $m_h$ denote a solution of Problem~\ref{prob:semi}. 
Then 
\begin{align*}
&a \|\pi_h^{k-1} p_h(t)\|^2 + b \|\pi_h^{k-1} m_h(t)\|^2 \\ 
&\le  C(T) \left(a \|\pi_h^{k-1} p_h(0)\|^2 + b\|\pi_h^{k-1} m_h(0)\|^2 
    + \int_0^t \frac{1}{a}\|\pi_h^{k-1} f(s)\|^2 + \frac{1}{b}\|\pi_h^{k-1} g(s)\|^2 ds \right)\notag
\end{align*}
with constant $C(T) \le C T$ and $C$ independent of $T$ and the solution. 
\end{lemma}
\begin{proof}
Let us start by noting that $(\pi_h^{k-1} q_h,\pi_h^{k-1} q) = (q_h, \pi_h^{k-1} q)$ for all $q \in H^1(x_L,x_R)$.
By testing \eqref{eq:var1h}--\eqref{eq:var2h} with $q_h = \pi_h^{k-1} p_h(t)$ and $v_h = \pi_h^{k-1} m_h(t)$, we then get
\begin{align*}
&\ddt \left(\frac{a}{2}  \|\pi_h^{k-1} p_h(t)\|^2 + \frac{b}{2} \|\pi_h^{k-1} m_h(t)\|^2\right) \\
&= (a \dt p_h(t) , \pi_h^{k-1}p_h(t)) + (b \dt m_h(t) , \pi_h^{k-1} m_h(t)) \\
&= -(\dx m_h(t),\pi_h^{k-1}p_h(t)) - (\dx p_h(t), \pi_h^{k-1} m_h(t)) - (d \pi_h^{k-1} m_h(t), \pi_h^{k-1} m_h(t)) \\
&\qquad + (\pi_h^{k-1}f(t),\pi_h^{k-1}p_h(t)) + (\pi_h^{k-1} g(t),\pi_h^{k-1}m_h(t)).
\end{align*}
By identity \eqref{eq:commuting}, integration-by-parts, and the boundary conditions \eqref{eq:prob3}, 
one can verify that $(\dx m_h(t),\pi_h^{k-1}p_h(t)) + (\dx p_h(t), \pi_h^{k-1} m_h(t)) = 0$.
Via Cauchy-Schwarz and Young inequalities, and using positivity of $d$, we then obtain the estimate
\begin{align*}
&\ddt \left(\frac{a}{2}  \|\pi_h^{k-1} p_h(t)\|^2 + \frac{b}{2} \|\pi_h^{k-1} m_h(t)\|^2\right) \\
&=- (d \pi_h^{k-1} m_h(t), \pi_h^{k-1} m_h(t)) + (\pi_h^{k-1} f(t),\pi_h^{k-1}p_h(t)) + (\pi_h^{k-1} g(t),\pi_h^{k-1}m_h(t)) \\
&\le \frac{\alpha}{2} (a\|\pi_h^{k-1} p_h(t)\|^2 + b\|\pi_h^{k-1} m_h(t)\|^2) 
+ \frac{1}{2\alpha}(\frac{1}{a}\|\pi_h^{k-1}f(t)\|^2 + \frac{1}{b}\|\pi_h^{k-1} g(t)\|^2).
\end{align*}
The Gronwall lemma and the choice $\alpha=1/T$ finally yields the assertion.
\end{proof}
Note that the above estimate does not yet give full control over the solution. 
A repeated application, however, allows us to prove the following stability estimate.
\begin{lemma} \label{lem:stability}
Let $p_h$, $m_h$ denote a solution of Problem~\ref{prob:semi}.
Then 
\begin{align*}
&\|p_h(t) \|^2 + \|m_h(t)\|^2 \\
&\le C'(T) \Big( \|\pi_h^{k-1}p_h(0)\|^2+ \|\pi_h^{k-1}m_h(0)\|^2 + h \|\pi_h^{k-1}\dt p_h(0)\|^2
  + h\|\pi_h^{k-1}\dt m_h(0)\|^2 \\
&\qquad \ + \int_0^t \|\pi_h^{k-1} f(s)\|^2 + \|\pi_h^{k-1}  g(s)\|^2 ds 
         +h \|\pi_h^{k-1} \dt f(s)\|^2 + h\|\pi_h^{k-1} \dt g(s)\|^2 \Big)
\end{align*}
for all $0 \le t \le T$ with $C'(T)=C' T$ and $C'$ independent of $T$ and of the solution.
\end{lemma}
\begin{proof}
As a direct consequence of the Poincar\'e inequality, one has 
\begin{align*}
\|p_h\| \le \|\pi_h^{k-1} p_h\| + h\|\dx p_h\|
\quad \text{and} \quad 
\|m_h\| \le \|\pi_h^{k-1} m_h\| + h\|\dx m_h\|. 
\end{align*}
The first terms in these estimates are already covered by Lemma~\ref{lem:estimate}. 
From the two equations \eqref{eq:var1h}--\eqref{eq:var2h} with $\widetilde q_h = \dx m_h(t)$ 
and $\widetilde v_h = \dx p_h(t)$, we further deduce that 
\begin{align*}
\|\dx m_h(t)\|^2 
&\le (\|\pi_h^{k-1}f(t)\| + a \|\pi_h^{k-1} \dt p_h(t)\|) \|\dx m_h(t)\| \qquad \text{and}\\
\|\dx p_h(t)\|^2 
&\le (\| \pi_h^{k-1} g(t) \| + b \|\pi_h^{k-1} \dt m_h(t) \| + \overline d \|\pi_h^{k-1} m_h(t)\|) \|\dx p_h(t)\|.  
\end{align*}
Bounds for $\|\pi_h^{k-1}\dt p_h(t)\|$ and $\|\pi_h^{k-1} \dt m_h(t)\|$ can be obtained by formally 
differentiating \eqref{eq:var1h}--\eqref{eq:var2h} with respect to time and applying Lemma~\ref{lem:estimate} 
for the resulting system. 
A combination of the above estimates then yields the assertion of the lemma. 
\end{proof}
%
%
\begin{remark} \label{rem:wellposedh}
Let us note that Problem~\ref{prob:semi} formally amounts to a finite dimensional system of differential algebraic equations. From the stability estimates of Lemma~\ref{lem:stability} and \cite[Theorem~4.12]{KunkelMehrmann06},
one can deduce that this system is solvable for any choice of admissible initial values. The semidiscretization is thus well-defined. 
\end{remark}

\section{Error estimates}
\label{sec:errorestimates}

We start by decomposing the error via
$\|p-p_h\| \le \|p-I_h^k p\| + \|I_h^k p - p_h\|$  and $\|m-m_h\| \le \|m-I_h^k m\| + \|I_h^k m - m_h\|$
into approximation and discrete error components. 
For the first part, we can utilize the following well-known estimates.
\begin{lemma} \label{lem:approx}
Let $w \in H^{s+1}(T_h)$, $0 \le s \le k$. Then 
\begin{align} 
\|w - I_h^{k-1} w\| \le h^{s+1} |w|_{s+1;h}. 
\end{align}
For any $w \in L^2(x_L,x_R) \cap H^{s}(T_h)$, $0 \le s \le k$, one has
\begin{align}
\|w - \pi_h^{k-1} w\| \le h^s |w|_{s;h}. 
\end{align}
Here $H^s(T_h) = \{w \in L^2(x_L,x_R) : w|_T \in H^s(T)\}$ is the space of piecewise smooth functions
and $|w|_{s;h}:=(\sum_T \|\dx^{s} w\|_{L^2(T)}^2)^{1/2}$ is the corresponding seminorm.
\end{lemma}
Using equations \eqref{eq:var1}--\eqref{eq:var2} and \eqref{eq:var1h}--\eqref{eq:var2h} 
characterizing the continuous and the discrete solutions, one can see that 
the discrete error components $\widehat p_h(t) := I_h^k p(t) - p_h(t)$ and $\widehat m_h(t):=I_h^k m(t)-m_h(t)$ satisfy equations \eqref{eq:var1h}--\eqref{eq:var2h} with initial values $\widehat p_h(0)=0$ and $\widehat m_h(0)=0$, and right hand sides given by
\begin{align*}
\widehat f(t) &:= a (I_h^{k} \dt p(t) - \dt p(t))
\quad \text{and} \quad \\
\widehat g(t) &:= b (I_h^{k} \dt m(t) -\dt  m(t)) + d(\pi_h^{k-1} I_h^k m(t) - m(t)).  
\end{align*}
By the a-priori estimates of Lemma~\ref{lem:stability}, one then obtains the following result.
\begin{lemma} \label{lem:discrete}
Let $d \in P_0(T_h)$ be piecewise constant. Then for all $0 \le t \le T$ one has
\begin{align*} 
&\|I_h^k p(t) - p_h(t)\|^2 + \|I_h^k m(t) - m_h(t)\|^2 \\ 
&\le C''(T) \Big( h\|I_h^k \dt p(0) - \dt p(0)\|^2 + h\|I_h^k \dt m(0) - \dt m(0)\|^2 \\
&\qquad + \int_0^t \|I_h^k m(s) - m(s)\|^2 + \|I_h^k \dt p(s) - \dt p(s)\|^2 + \|I_h^k \dt m(s) - \dt m(s)\|^2 \\
&\qquad \qquad \qquad \qquad 
+ h \|I_h^k \dtt p(s) - \dtt p(s)\|^2 + h \|I_h^k \dtt m(s) - \dtt m(s)\|^2 ds\Big),
\end{align*}
with a constant $C''(T) = C'' T$ and $C''$ independent of $T$ and of the solution. 
\end{lemma}
\begin{proof}
We apply Lemma~\ref{lem:stability} for $\widehat p_h(t) = I_h^k p(t) - p_h(t)$ and $\widehat m_h(t) = I_h^k m(t) - m_h(t)$ and then estimate the terms on the right hand side of the result step by step.
By definition of the initial values, we have $\widehat p_h(0)=\widehat m_h(0)=0$. 
Moreover, 
\begin{align*}
\pi_h^{k-1} \dt p_h(0) 
&= \pi_h^{k-1} f(0) -\dx m_h(0) 
= \pi_h^{k-1} f(0) -\dx I_h^k m(0) \\
&= \pi_h^{k-1} f(0) -\pi_h^{k-1} \dx m(0) 
= \pi_h^{k-1} \dt p(0),
\end{align*}
where we used the definition of the initial value $m_h(0)$ in the second and \eqref{eq:commuting} in the third step.
Thus $\|\pi_h^{k-1} \dt \widehat p_h(0)\| \le \|I_h^k \dt p(0) - \dt p(0)\|$, and in a similar manner, one can show 
$\|\pi_h^{k-1} \dt \widehat m_h(0)\| \le \|I_h^k \dt m(0) - \dt m(0)\|$.
This explains the first two terms in the estimate in the lemma. 
The terms under the integral are derived by estimating $\|\pi_h^{k-1} \widehat f(t)\|$, 
$\|\pi_h^{k-1} \widehat g(t)\|$ and the derivatives $\|\pi_h^{k-1} \dt \widehat f(t)\|$, $\|\pi_h^{k-1} \dt \widehat g(t)\|$ via the triangle inequality, and noting that 
\begin{align*}
\pi_h^{k-1} (d \pi_h^{k-1} I_h^k m(t) - d m(t)) = d \pi_h^{k-1} (I_h^k m(t) - m(t)),
\end{align*}
where we used that $d$ is piecewise constant. 
\end{proof}
\begin{remark}
A similar result can be proven for piecewise smooth $d \in W^{1,\infty}(T_h)$
and additional terms of the form $\|d-\pi_h^0 d\| \|\pi_h^{k-1} p(t) - p(t)\|$ arise. 
For $d \in W^{1,\infty}(T_h)$, the product of the two terms again has optimal approximation order. 
\end{remark}
By combination of the above estimates, we finally obtain the following result.
\begin{theorem} \label{thm:convergence}
Let (A1)--(A2) hold and $d \in W^{1,\infty}(T_h)$. Furthermore, let $(p,m)$ be a sufficiently smooth 
solution of \eqref{eq:prob1}--\eqref{eq:prob3}. Then for all $0 \le t \le T$, one has 
\begin{align*}
\|p(t) - p_h(t)\| + \|m(t) - m_h(t)\| \le C(u,p,T) h^{k+1}.
\end{align*}
\end{theorem}
For sufficiently smooth solutions, the proposed method thus yields convergence with the optimal order
that can be expected.

\section{Extension to networks}
\label{sec:networks}

We now illustrate that our method and the convergence results of the previous section can be 
generalized easily to pipe networks. Let $(\V,\E)$ denote a directed graph with vertices $v \in \V$ 
and edges $e \in \E$; see Figure~\ref{fig:network} for illustration.
\begin{figure}[ht!]
\centering
\begin{tikzpicture}[scale=.5]
\small
\node[circle,draw,inner sep=2pt] (v1) at (0,2) {$v_1$};
\node[circle,draw,inner sep=2pt] (v2) at (3,2) {$v_2$};
\node[circle,draw,inner sep=2pt] (v3) at (6,4) {$v_3$};
\node[circle,draw,inner sep=2pt] (v4) at (6,0) {$v_4$};
\node[circle,draw,inner sep=2pt] (v5) at (9,2) {$v_5$};
\node[circle,draw,inner sep=2pt] (v6) at (12,2) {$v_6$};
\draw[->,thick,line width=1.5pt] (v1) -- node[above] {$e_1$} ++(v2);
\draw[->,thick,line width=1.5pt] (v2) -- node[above,sloped] {$e_2$} ++(v3);
\draw[->,thick,line width=1.5pt] (v2) -- node[above,sloped] {$e_3$} ++(v4);
\draw[->,thick,line width=1.5pt] (v3) -- node[above,sloped,rotate=180] {$e_4$} ++(v4);
\draw[->,thick,line width=1.5pt] (v3) -- node[above,sloped] {$e_5$} ++(v5);
\draw[->,thick,line width=1.5pt] (v4) -- node[above,sloped] {$e_6$} ++(v5);
\draw[->,thick,line width=1.5pt] (v5) -- node[above,sloped] {$e_7$} ++(v6);
\end{tikzpicture}
\caption{\label{fig:network}Directed graph $(\V,\E)$ modeling the pipe network topology used for numerical tests.}
\end{figure}
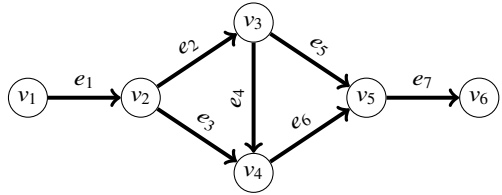
For any edge $e=(v_1,v_2)$, we define $n^e(v_1)=-1$ and $n^e(v_2)=1$. The matrix $N$ with entries $N_{ij}=n_{ej}(v_i)$ then is the incidence matrix of the graph.
For any vertex $v \in \V$, we define $\E(v)=\{e : e = (v,\cdot) \text{ or } e=(\cdot,v)\}$,
and we set $V_0 = \{v \in \V : |\E(v)|>1\}$ and $\V_\partial = \{v \in \V : |\E(v)|=1\}$ which gives a decomposition $\V = \V_0 \cup \V_\partial$ into interior and boundary vertices. 

To every edge $e$, we associate a positive length $\ell^e$, and we identify $e$ with 
$[0,\ell^e]$ in the sequel. This allows us to define spaces 
$L^p(\E) = \{ v : v|_e \in L^p(e)\}$ and $H^1(\E) = \{v \in L^p(\E) : v|_e \in H^1(e)\}$
of, respectively, integrable and piecewise smooth functions on the graph. 
The flow of gas in a pipe network is then described as follows: On every edge $e$ representing a pipe,
we require that
\begin{align} 
a^e \dt p^e + \dx m^e &=  f^e              \label{eq:network1}\\
b^e \dt m^e + \dx p^e + d^e m^e &= g^e,    \label{eq:network2}
\end{align}
where $f^e = f|_e$ denotes the restriction of a function $f \in L^p(\E)$ to one edge. 
The equations for the individual pipes are coupled by algebraic conditions 
\begin{alignat}{2}
\sum\nolimits_{e \in \E(v)} m^e(v) n^e(v) &= 0 \qquad &&v \in \V_0 \label{eq:network3}\\
p^e(v) &= p^{e'}(v)  \qquad && v \in \V_0, \ e,e' \in \E(v)       \label{eq:network4}
\end{alignat}
at the pipe junctions, and at the boundary vertices, we assume that 
\begin{align}
p^e(v) &= 0 \qquad v \in \V_\partial. \label{eq:network5}
\end{align}
Inhomogeneous and other types of boundary conditions can again be incorporated with minor modifications.
For the analysis of the problem, we now utilize the spaces
\begin{align}
  H^1_0 &:= \{p \in H^1(\E) : \eqref{eq:network4} \text{ and } \eqref{eq:network5} \text{ are valid}\}  \\
H(\div) &:= \{m \in H^1(\E) : \eqref{eq:network3} \text{ is valid}\} 
\end{align}
which are the natural generalization of those used for the analysis on a single pipe.
Any solution $(p,m)$ of \eqref{eq:network1}--\eqref{eq:network5} then again satisfies 
$p(t) \in H^1_0$, $m(t) \in H(\div)$, and 
\begin{align}
(a \dt p(t), \widetilde q) + (\dx m(t),\widetilde q ) &= (f(t),\widetilde q) \label{eq:var1n}\\
(b \dt m(t), \widetilde v) + (\dx p(t),\widetilde v ) + (d m(t), \widetilde v) &= (g(t),\widetilde q) \label{eq:var2n}
\end{align}
for all $\widetilde q \in L^2(\E)$, $\widetilde v \in L^2(\E)$, and all $0 < t < T$. 
Here $(v,w) = \sum_e (v^e,w^e)_e$ with $(v^e,w^e)_e = \int_e v^e w^e dx$ denotes the scalar product on $L^2(\E)$.

\begin{remark}
Let us note that \eqref{eq:var1n}--\eqref{eq:var2n} has exactly the same form as the variational problem \eqref{eq:var1h}--\eqref{eq:var2h} on a single pipe. The inexact Petrov-Galerkin method and all results 
derived in the previous sections therefore translate almost verbatim to the network setting;
let us refer to \cite{EggerKugler18} for details and similar results for a different method, and to Section~\ref{sec:numerics} for numerical illustration.
\end{remark}

\section{Remarks on the efficient implementation}
\label{sec:implementation}

In the discretization of \eqref{eq:var1n}--\eqref{eq:var2n}, also compare with \eqref{eq:var1h}--\eqref{eq:var2h},
the continuity and boundary conditions \eqref{eq:network3}--\eqref{eq:network5} are directly incorporated in 
the definition of the spaces $Q_h \subset H^1_0$ and $V_h \subset H(\div)$. 
For the implementation, it may be more convenient to use larger spaces
$Q_h,V_h \subset H^1(\E)$, and to enforce some of the boundary and coupling conditions 
\eqref{eq:network3}--\eqref{eq:network5} explicitly by additional equations. 
Using the wording of \cite{BoffiBrezziFortin13}, this approach of relaxing continuity conditions 
might be called {\em hybridization}.
Since the resulting method is algebraically equivalent to the original scheme based on 
function spaces with incorporated coupling and boundary conditions, all results 
of the previous sections apply verbatim also to the method obtained after hybridization.

\section{Nonlinear problems}
\label{sec:nonlinear}

The formal extension of the Petrov-Galerkin method to nonlinear problems is straight-forward. 
The discrete variational formulation for \eqref{eq:euler1}--\eqref{eq:euler2}, for instance,
reads
\begin{align*} 
(A \dt \rho_h(t),\widetilde q_h) + (\dx m_h(t),\widetilde q_h) &= 0 \\
(\dt m_h(t),\widetilde v_h) + (\dx \left(\frac{m_h(t)^2}{A \rho_h(t)} + A p_h(t)\right),\widetilde v_h) &= -(\frac{\lambda}{2D} \frac{|m_h(t)|}{A \rho_h(t)} \pi_h^{k-1} m_h(t), \widetilde v_h).
\end{align*}
Numerical quadrature can be used in practice to facilitate the handling of the nonlinear terms. 
We do not give a complete convergence analysis here, but instead, we will demonstrate by numerical 
tests that for smooth solutions, the convergence results of Theorem~\ref{thm:convergence} remain valid, at least in the practically relevant case of nonlinear friction.

\section{Numerical results}
\label{sec:numerics}

We now illustrate the theoretical results of 
Section~\ref{sec:errorestimates} by numerical tests.
For our computations, we consider the pipe network depicted in Figure~\ref{fig:network}. 
As a first test case, we consider the linear problem \eqref{eq:network1}--\eqref{eq:network4} with 
inhomogeneous boundary conditions 
\begin{align}
p|_v(t) = p_v(t) \qquad v \in \V_\partial 
\end{align}
and we set $p_{v_1}(t)=1$ and $p_{v_6}(t)=1+\frac{1}{2}\sin(\pi t)$ in the following. 
All pipes are chosen of unit length $\ell=1$ and the model parameters are set to $a \equiv b \equiv d \equiv 1$.
The simulation is started from a stationary state for the boundary values at initial time.
The results of the computations are summarized in the left column of Table~\ref{tab:1}.
As predicted by our theoretical results, we observe second order convergence. 

\vspace*{-1em}

\begin{table}[ht]
\caption{Errors $e_h=(a\|p_h(T) - p_{h/2}(t)\|^2 + b\|m_h(T) - m_{h/2}(T)\|^2)^{1/2}$ at time $T=10$ obtained with the Petrov-Galerkin approximation for the network problem with different gas flow models: linear model (left), semilinear model (middle), and quasilinear model (right).}
\label{tab:1}       
\setlength\tabcolsep{2ex}
\centering
\begin{tabular}{c||c|c||c|c||c|c}
$h$      & linear   & eoc  & semilinear & eoc  & quasilinear & eoc \\
\hline
0.10000  & 0.01936  & --   & 0.02359    & --   & 0.02534     & --   \\
0.05000  & 0.00482  & 2.00 & 0.00660    & 1.83 & 0.00693     & 1.87 \\
0.02500  & 0.00120  & 2.00 & 0.00168    & 1.97 & 0.00200     & 1.79 \\
0.01250  & 0.00030  & 2.00 & 0.00042    & 1.99 & 0.00076     & 1.40 \\
0.00625  & 0.00008  & 2.00 & 0.00011    & 2.00 & 0.00036     & 1.09 \\[-1.5em]
\end{tabular} 
\end{table}
We now repeat our numerical tests for the same network but with a semilinear gas flow model resulting from  \eqref{eq:euler1}--\eqref{eq:euler3} by dropping the nonlinear term $\dx (\frac{m^2}{A \rho})$ in equation \eqref{eq:euler2}. The model parameters are chosen as $A=1$, $c=1$, and $\lambda/(2D)=7/2$; the latter was selected such that average of the resulting mass flow was similar to that of the linear model considered above. The computational results are depicted in the middle column of Table~\ref{tab:1}. Also for this nonlinear friction model, we observe second order convergence. These results can be explained theoretically in a similar way as those for the linear case by using a perturbation argument. 
In the right column of Table~\ref{tab:1}, we display the corresponding results for the quasilinear flow model 
\eqref{eq:euler1}--\eqref{eq:euler3} with the same parameters as used in the semilinear case. Note that a decrease in the convergence rates to first order is observed here. This is no surprise, since our analysis heavily relied on the anti-symmetry of the spatial derivative terms in \eqref{eq:var1h}--\eqref{eq:var2h}, which is no longer valid for the quasilinear model \eqref{eq:euler1}--\eqref{eq:euler2}. 

In Figure~\ref{fig:2}, we display the flow rates $m|_{v}$ at the boundary vertices $v_1$ and $v_6$
for the three different gass flow models discussed above as function function of time. 
\begin{figure}[ht!]
\centering
\vspace*{-1em}
\includegraphics[width=0.95\textwidth]{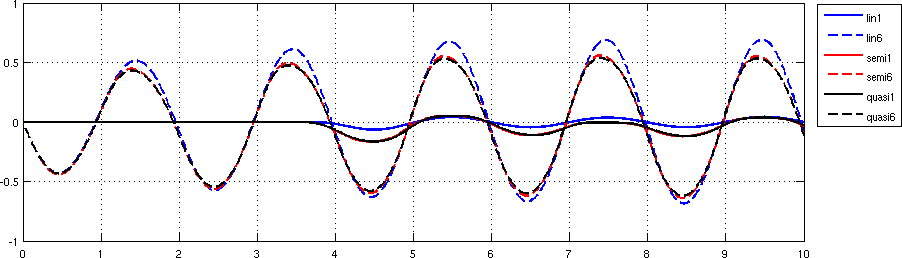} 
\caption{\label{fig:2}Flow rates at boundary vertices $v_1$ and $v_6$ for linear, semilinear, and quasilinear 
flow models.}
\end{figure}
The results are in reasonable agreement. In summary, the semilinear model seems to yield the best compromise between modelling errors and convergence order.

\begin{acknowledgement}
Support by the German Science Foundation (DFG) via grants TRR~154, TRR~146, GSC~233, and Eg-331/1-1 is 
gratefully acknowledged.
\end{acknowledgement}

\vspace*{-2em}


\begin{thebibliography}{10}
\providecommand{\url}[1]{{#1}}
\providecommand{\urlprefix}{URL }
\expandafter\ifx\csname urlstyle\endcsname\relax
  \providecommand{\doi}[1]{DOI~\discretionary{}{}{}#1}\else
  \providecommand{\doi}{DOI~\discretionary{}{}{}\begingroup
  \urlstyle{rm}\Url}\fi

\bibitem{BoffiBrezziFortin13}
Boffi, D., Brezzi, F., Fortin, M.: Mixed finite element methods and
  applications, \emph{Springer Series in Computational Mathematics}, vol.~44.
\newblock Springer, Heidelberg (2013)

\bibitem{BrouwerGasserHerty11}
Brouwer, J., Gasser, I., Herty, M.: Gas pipeline models revisited: Model
  hierarchies, non-isothermal models and simulations of networks.
\newblock Multiscale Model. Simul. \textbf{9}, 601--623 (2011)

\bibitem{Egger18}
Egger, H.: A robust conservative mixed finite element method for compressible
  flow on pipe networks.
\newblock SIAM J. Sci. Comput. \textbf{40}, A108--A129 (2018)

\bibitem{EggerKugler18}
Egger, H., Kugler, T.: Damped wave systems on networks: {E}xponential stability
  and uniform approximations.
\newblock Numer. Math. \textbf{138}, 839--867 (2018)

\bibitem{Kiuchi94}
Kiuchi, T.: An implicit method for transient gas flows in pipe networks.
\newblock Int. J. Heat and Fluid Flow \textbf{15}, 378--383 (1994)

\bibitem{KunkelMehrmann06}
Kunkel, P., Mehrmann, V.: Differential-algebraic equations: Analysis and
  numerical solution.
\newblock EMS Textbooks in Mathematics. European Mathematical Society (EMS),
  Z\"urich (2006)

\bibitem{MarcatiMilani90}
Marcati, P., Milani, A.: The one-dimensional {D}arcy's law as the limit of a
  compressible {E}uler flow.
\newblock J. Diff. Equat. \textbf{84}, 129--147 (1990)

\bibitem{Osiadacz84}
Osiadacz, A.: Simulation of transient gas flows in networks.
\newblock Int. J. Numer. Meth. Fluids \textbf{4}, 13--24 (1984)

\bibitem{RachfordDupont74}
Rachford~Jr., H.H., Dupont, T.: A fast, highly accurate means of modeling
  transient flow in gas pipeline systems by variational methods.
\newblock J. Soc. Petrol. Eng. \textbf{14}, 165--178 (1974).
\newblock \doi{doi.org/10.2118/4005-A}.
\newblock SPE-4005-A

\bibitem{Reigstad14}
Reigstad, G.A.: Numerical network models and entropy principles for isothermal
  junction flow.
\newblock Netw. Heterog. Media \textbf{9}, 65--95 (2014)

\end{thebibliography}

\end{document}